\documentclass[12pt]{amsart}
\usepackage{mathtools, amssymb, amsfonts}
\usepackage{amsthm}
\usepackage{hyperref}

\newtheorem{theorem}{Theorem}[section]
\newtheorem{proposition}[theorem]{Proposition}
\newtheorem{lemma}[theorem]{Lemma}
\newtheorem{corollary}[theorem]{Corollary}
\theoremstyle{remark}
\newtheorem{remark}[theorem]{Remark}

\newcommand{\CH}{\mathbb{CH}}
\newcommand{\Teich}{\mathcal T}
\newcommand{\cX}{\mathcal X}
\newcommand{\PU}{\mathrm{PU}}
\newcommand{\SU}{\mathrm{SU}}
\newcommand{\tr}{\operatorname{tr}}
\newcommand{\im}{\operatorname{im}}
\newcommand{\Hess}{\operatorname{Hess}}
\newcommand{\II}{\mathrm{II}}
\newcommand{\dd}{\, \mathrm d}

\title[Critical-point-free energy]{Critical-point-free energy for\\
fractional-Toledo representations}

\author[R. Bardhan]{Rivu Bardhan}
\address{Department of Mathematics,  Shiv Nadar University,  Dadri 201314, 
Uttar Pradesh,  India}
\email{rb212@snu.edu.in}

\author[A. Dhochak]{Anu Dhochak}
\address{International Centre for Theoretical Sciences,  Tata Institute of
Fundamental Research,  Bengaluru 560089,  India}
\email{anu.dhochak@icts.res.in}

\author[P. Kumar]{Pradip Kumar}
\address{Department of Mathematics,  Shiv Nadar University,  Dadri 201314, 
Uttar Pradesh,  India}
\email{pradip.kumar@snu.edu.in}
\date{}

\begin{document}

\begin{abstract}
Let $S_g$ be a closed oriented surface of genus $g\ge2$.  For a reductive
representation $\rho:\pi_1(S_g)\to\PU(2, 1)$,  let $E_\rho$ be the energy
function on Teichm\"uller space associated to equivariant harmonic maps into
$\CH^2$.  For every positive integer $d$ with $3\nmid d$,  all sufficiently
large $h$,  and every $g>h$,  we construct an irreducible reductive
representation
\[
\rho_{g, h, d}:\pi_1(S_g)\to\PU(2, 1)
\]
with
\[
\tau(\rho_{g, h, d}) = 2h-2-\frac{2d}{3}\notin\mathbb Z, 
\qquad
\operatorname{Crit}(E_{\rho_{g, h, d}}) = \varnothing.
\]
Consequently,  the associated branched-minimal-surface forgetful map is not
surjective in these nonintegral Toledo components.
\end{abstract}

\maketitle

\section{Introduction}

Let $S_g$ be a closed oriented surface of genus $g \ge 2$,  and equip
$\CH^2$ with the symmetric metric with sectional curvatures in
$[-4, -1]$. 
If $\rho:\pi_1(S_g) \to \PU(2, 1)$ is reductive and
$J \in \Teich(S_g)$,  Corlette's theorem
\cite[Corollary~3.5]{Corlette} gives a $\rho$-equivariant harmonic map
\[
 f_{\rho, J}:(\widetilde S_g, J)\longrightarrow\CH^2.
\]
When $\rho$ is irreducible,  this map is nonconstant and cannot have image in a
geodesic; it is therefore unique by the equivariant uniqueness theorem
\cite[Theorem~3.2]{DaskalopoulosWentworth}.  The energy associated to this harmonic map defines a
differentiable function $E_\rho$ on Teichm\"uller space.  The first-variation
formula \cite[Theorem~3.11]{DaskalopoulosWentworth} of energy functional says that $J$ is critical if and only if the Hopf differential of $f_{\rho, J}$ vanishes,  equivalently, 
when $f_{\rho, J}$ is a possibly branched minimal map.

We call a representation in $\PU(2,  1)$ \emph{irreducible} if its image
preserves no proper complex projective subspace of $\mathbb{CP}^2$.
Equivalently,  its image fixes no point of $\mathbb{CP}^2$.  According to the
sign of a representative line for such a point,  a reducible subgroup fixes a
point of $\CH^2$,  fixes a point of $\partial\CH^2$,  or preserves a complex
geodesic.
We construct irreducible representations with nonintegral Toledo invariant
whose energy functions have no critical points. We prove the following

\begin{theorem}\label{thm:main}
Fix a positive integer $d$ with $3\nmid d$. There exists $h_0(d)$ such that, 
for every $h\ge h_0(d)$ and every $g>h$,  there is an irreducible reductive
representation
\[
 \rho_{g, h, d}:\pi_1(S_g)\longrightarrow\PU(2, 1)
\]
satisfying
\begin{equation}\label{eq:toledo-main}
 \tau(\rho_{g, h, d}) = 2h-2-\frac{2d}{3}\notin\mathbb Z, 
 \qquad
 \operatorname{Crit}(E_{\rho_{g, h, d}}) = \varnothing.
\end{equation}
Branch points are allowed in this assertion.
\end{theorem}

We start the construction using Bronstein's almost-Fuchsian representation
$\rho_0:\pi_1(S_h) \to \PU(2,  1)$ and precompose it with the epimorphism on
fundamental groups induced by an orientation-preserving degree-one pinch
$P:S_g\to S_h$. Thus $\rho_{g, h, d} = \rho_0\circ P_*$. The image remains the same,  and hence this representation is
irreducible,  but the new representation has a large kernel.

Goldman--Wentworth prove that $E_\sigma$ is proper when $\sigma$ is a
convex-cocompact representation \cite[Theorem~A]{GoldmanWentworth}. This requires $\sigma$ to be an isomorphism of the surface group
onto a convex-cocompact discrete subgroup. So,  for a convex-cocompact representation,  $E_{\rho_0}$ is
proper on $\Teich(S_h)$ and attains a minimum. Although
$\rho  =  \rho_0 \circ P_*$ has exactly the same image subgroup as $\rho_0$,  it has
nontrivial kernel and is not convex cocompact as a representation of
$\pi_1(S_g)$. Our conclusion shows that,  for our $\rho$,  $E_\rho$ has no critical point at all.

Fractional-Toledo representations admitting equivariant minimal surfaces are already known
\cite{BronsteinPU,  LoftinMcIntosh}. The obstruction we show here is the combination of a
`geometric barrier' and the degree-one pinch; we refer to it as a ``trapping lemma.''
In equivalent terms,  Theorem~\ref{thm:main} produces points of the
$\PU(2,  1)$ character variety outside the image of the forgetful map from
equivariant branched minimal surfaces. This problem was raised in
\cite[Section~6.3]{LoftinMcIntosh}.

\section{Energy,  Hopf differentials,  and Toledo invariant}
\subsection{The energy function}
Fix an irreducible reductive representation
$\rho:\pi_1(S_g)\to\PU(2,  1)$.  For $J\in\Teich(S_g)$,  let $f_{\rho,  J}$ be
the Corlette harmonic map.  Since we take the irreducible representation,  therefore the displacement function of the action on
$\CH^2$ is proper. This can be seen directly. Suppose it is not proper,  then  there exists an escaping sequence on which the displacements
of a finite generating set remain bounded,  which would determine a common fixed
point of $\im\rho$ on $\partial\CH^2$,  contrary to irreducibility.  Thus the following results apply even when $\rho$ has a kernel.  We use
the convention
\[
 E_\rho(J)  =  \frac12 \int_{S_g}|\dd f_{\rho,  J}|^2 \,  \dd A_J.
\]
The function is independent of the metric in the conformal
class.  The Hopf differential
$$q_{\rho, J}  =  \big\langle (f_{\rho,  J})_z,  (f_{\rho,  J})_z \big\rangle\, \dd z^2 \in H^0(S_g,  K_J^2)$$
is holomorphic.  The first variation is a nonzero multiple of the
real part of the Beltrami-- quadratic-differential pairing with $q_{\rho,  J}$.

Since the Beltrami--quadratic-differential pairing is nondegenerate, 
\begin{equation}\label{eq:critical-equivalence}
 J \in \operatorname{Crit}(E_\rho)
 \quad \Longleftrightarrow \quad q_{\rho,  J}  =  0
 \quad \Longleftrightarrow\quad f_{\rho,  J}
 \text{ is weakly conformal}.
\end{equation}
Since $\rho$ is irreducible,  the harmonic map is nonconstant.  A nonconstant
weakly conformal harmonic map from a surface is a branched minimal immersion;
its branch set is discrete.

\subsection{Normalization and sign of the Toledo invariant}

All complex structures on $S_g$ are compatible with its fixed orientation.
Let $J_{\CH^2}$ and $g_{\CH^2}$ be the complex structure and metric on
$\CH^2$,  and take the positive K\"ahler form
\[
 \omega(u, v) = g_{\CH^2}(J_{\CH^2}u, v).
\]
For any smooth $\rho$-equivariant map
$f:\widetilde S_g\to\CH^2$,  the invariant form $f^*\omega$ descends to
$S_g$,  and we set
\begin{equation}\label{eq:toledo-def}
 \tau(\rho) = \frac{2}{\pi}\int_{S_g}f^*\omega.
\end{equation}
The number in \eqref{eq:toledo-def} is independent of $f$, 
and
\begin{equation}\label{eq:toledo-lattice}
 \tau(\rho) \in \frac23 \mathbb Z, 
 \qquad |\tau(\rho)| \le  2g-2.
\end{equation}
The values index the connected components of the $\PU(2,  1)$ character
variety \cite{GKL, Toledo1989, Xia}.  Under this normalization,  a representation
lifts to $\SU(2,  1)$ only  when its Toledo invariant belongs to
$2\mathbb Z$ \cite[Remark~2.2]{LoftinMcIntosh}. By
\emph{nonintegral Toledo invariant},  we  mean
\[
 \tau(\rho)  =  \frac{2k}{3}\quad\text{with}\quad 3\nmid k.
\]

If $\kappa$ is an antiholomorphic isometry of $\CH^2$ and
$\rho^\kappa(\gamma)  =  \kappa\rho(\gamma)\kappa^{-1}$,  then
\begin{equation}\label{eq:toledo-conjugation}
 \tau(\rho^\kappa)  =  -\tau(\rho), 
\end{equation}
because $\kappa^*\omega = -\omega$.

\begin{lemma}\label{lem:naturality}
Let $P : S_g \to S_h$ be a continuous map of signed degree $m$ between the
oriented surfaces,  and let
$\rho_0 :\pi_1(S_h) \to \PU(2,  1)$.  Then
$$ \tau(\rho_0 \circ P_*)  =  m\, \tau(\rho_0).$$
\end{lemma}

\begin{proof}
We replace $P$ by a smooth representative.  If $F : \widetilde S_h \to \CH^2$ is $\rho_0$- equivariant,  then $F^*\omega$ is
$\pi_1(S_h)$- invariant and descends to a two-form $\alpha_0$ on $S_h$.
Moreover,  $F \circ \widetilde P$ is $(\rho_0 \circ P_*)$-equivariant,  and its
pulled-back K\"ahler form descends to $P^*\alpha_0$.  Hence
\[
 \int_{S_g}P^*\alpha_0  =  m\int_{S_h}\alpha_0.
\]
Now,  by multiplying by $2/\pi$,  we get the formula.
\end{proof}

\section{The almost-Fuchsian input and the representation}

\subsection{Bronstein's fractional-Toledo representations}
Equip $\CH^2$ with its standard K\"ahler metric $g_{\CH^2}$,  normalized so
that its sectional curvatures lie in $[-4,  -1]$.  If $J_{\CH^2}$ is its complex
structure,  then
\[
 \omega(u, v) = g_{\CH^2}(J_{\CH^2}u, v)
\]
is the positive K\"ahler form used in \eqref{eq:toledo-def}.  We take all norms and
distances in this subsection  with respect to $g_{\CH^2}$.

For an immersed surface $F:\Sigma \to \CH^2$ and a unit normal vector $\nu$, 
let $A_\nu$ be the shape operator defined by
\[
 \langle A_\nu u,  v \rangle  =  \langle \II_F(u,  v),  \nu\rangle, 
\]
and  we set
\begin{equation}\label{eq:II-operator-norm}
 \|\II_F\|_{\mathrm{op}}
  = \sup_{|\nu|  =  1}\|A_\nu\|_{\mathrm{op}}
  = \sup_{|u|  =  |v|  =  1}|\II_F(u,  v)|.
\end{equation}
This norm dominates this operator norm.  We call $F$
\emph{almost-Fuchsian} if its induced metric is complete and
$\sup_\Sigma\|\II_F\|_{\mathrm{op}} < 1$.

In \cite{BronsteinPU},  Bronstein uses the same normalization of the complex-hyperbolic metric as we do.  Theorem~6.1 \cite{BronsteinPU} gives an equivariant holomorphic immersion whose
second fundamental form is arbitrarily small.  The pointwise norm used there
and the operator norm defined in \eqref{eq:II-operator-norm} are uniformly equivalent, 
since the tangent and normal ranks are fixed.  In particular,  after decreasing the parameter Bronstein defined,  if necessary,  his estimate implies any prescribed
bound on $\|\II_F\|_{\mathrm{op}}$.  The embedding and convex-cocompactness
conclusions then follow from
\cite[Theorem~3.2 and Corollary~3.3]{BronsteinPU}.  Since a holomorphic curve
in a K\"ahler manifold is minimal,  these results give the following form of
Bronstein's construction.

We give the following theorem by Bronstein but in our normalization: 
\begin{theorem}\label{thm:bronstein}
Fix a positive integer $d$ and $0 < \eta < \tfrac12$.  There is
$h_0  =  h_0(d,  \eta)$ such that,  for every $h \ge h_0$,  there are a complex
structure $J_0$ on $ S_h $,  a convex-cocompact representation
\[
 \rho_0 : \pi_1(S_h)\longrightarrow\PU(2,  1), 
\]
and a $\rho_0$-equivariant holomorphic embedding
\(
  F : (\widetilde S_h,  J_0)\longrightarrow\CH^2
\)
with complete induced metric and
\(
 \sup_{\widetilde S_h}\|\II_F\|_{\mathrm{op}}\le\eta.
\)
The immersion $F$ is minimal and almost-Fuchsian,  and
\begin{equation}\label{eq:bronstein-toledo}
 \tau(\rho_0) = 2h-2-\frac{2d}{3}.
\end{equation}
\end{theorem}

We quickly explain the Toledo sign convention explicitly.  This is useful
because the convention for the Toledo invariant in \cite{BronsteinPU} is
opposite to the one used here. The discussion in the next three paragraphs is just an overview; the notations and terminologies are as in \cite{BronsteinPU}. 

In \cite[Theorem~5.2]{BronsteinPU},  for $h$ sufficiently large,  with the given condition,  we get a Riemann surface
$(S_h, J_0)$,  a holomorphic line bundle $N_h\to (S_h, J_0)$ of degree
\[
    \deg N_h = 3h-3+d, 
\]
and a nonzero holomorphic section
\(\alpha_h\in H^0(S_h, N_h).
\) Let $K_{J_0}$ denote the canonical bundle of $(S_h,  J_0)$.  The line bundle given by Bronstein is: 
\[
    L_h = K_{J_0}^{3}N_h^{-1}.
\]
We note that: 
\(K_{J_0}^{3}L_h^{-1}
      =  K_{J_0}^{3}\bigl(K_{J_0}^{3}N_h^{-1}\bigr)^{-1}
    \cong N_h.
\)
Thus the section $\alpha_h$ may equivalently be regarded as a section
\(\beta_h\in H^0\bigl(S_h, K_{J_0}^{3}L_h^{-1}\bigr), 
\)
This section $\beta_h$ appears in 
\cite[Theorem~3.7]{BronsteinPU} too. 

\smallskip 

These data determine a
$\rho_0$-equivariant holomorphic immersion
\[
    F:(\widetilde S_h, J_0) \longrightarrow \CH^2
\]
for a representation
\(
    \rho_0 : \pi_1(S_h) \longrightarrow \PU(2,  1).
\)  In Bronstein's notation \cite[equation~(3.5)]{BronsteinPU}  gives
\[
    \operatorname{Tol}_{\mathrm B}(\rho_0)
     =  -\frac23 \deg L_h.
\]
Using the degree computed above, 
\[
    \operatorname{Tol}_{\mathrm B}(\rho_0)
     = -\frac23(3h-3-d)
     = 2-2h+\frac{2d}{3}.
\]

To relate this formula to our convention,  recall the bundle underlying
Bronstein's projectively flat construction.  As a smooth bundle,  it has the
form
\[
    E_h = K_{J_0}^{-1}\oplus K_{J_0}L_h^{-1}\oplus\mathcal O, 
\]
with the Hermitian form of signature $(2, 1)$ positive on the first two
summands and negative on the last.  With the holomorphic structure determined
by the solution of the curvature equations,  the positive rank-two vector bundle
$V_h$ is represented by the extension
\begin{equation}\label{eq:bronstein-extension}
    0\longrightarrow K_{J_0}^{-1}
     \longrightarrow  V_h
    \longrightarrow K_{J_0}L_h^{-1}
    \longrightarrow 0.
\end{equation}
We have 
\(\deg V_h  =  \deg K_{J_0}^{-1}
      + \deg(K_{J_0}L_h^{-1})  = -\deg L_h.
\)

Our normalization of the Toledo invariant is
$\tau(\rho)  =  \frac{2}{\pi}\int_{S_h}F^*\omega.$ This is the convention used in
\cite[Section~2]{LoftinMcIntosh}. In this normalization,  the Chern--Weil formula for the positive rank-two summand is given by
$$\tau(\rho_0)  =  -\frac23 \deg V_h.$$
Using $\deg V_h  =  -\deg L_h$,  we therefore obtain $\tau(\rho_0)  =  2h-2-\frac{2d}{3}.$
Therefore, 
$$\tau(\rho_0) =  -\operatorname{Tol}_{\mathrm B}(\rho_0).$$
The map $F$ is holomorphic and our K\"ahler form is positive. Hence,  away from the zeros (where $F$ fails to be an immersion),  the pullback $F^*\omega$
is the positive area form of the induced metric on $(S_h, J_0)$. Thus
\( \int_{S_h}F^*\omega>0, \)
so our Toledo invariant must be positive. This agrees with
\(\tau(\rho_0) = 2h-2-\frac{2d}{3}>0\)
for the range of $h$ used in the construction.

\bigskip 

For the proof of Theorem~\ref{thm:main},  fix $\eta_0  =  \tfrac14$ and set
$h_0(d)  =  h_0(d,  \eta_0)$ so that
$3h-3-d > 0$ for $h \ge  h_0(d)$. Now we  apply Theorem~\ref{thm:bronstein} with
$\eta  =  \eta_0$ and put
\[
 Y  =  F(\widetilde S_h)\subset\CH^2.
\]
The surface $Y$ is complete,  embedded,  minimal,  and almost-Fuchsian.  Let
$NY$ be its normal bundle.  We define
$\exp^\perp(y, \xi)  =  \exp_y(\xi)$.  
\cite[Theorem~3.1]{BronsteinHadamard} gives a diffeomorphism
\begin{equation}\label{eq:normalexp}
 \exp^\perp : NY \longrightarrow \CH^2.
\end{equation}
In these normal coordinates,  the squared distance to $Y$ is
\begin{equation}\label{eq:psi}
 \Psi\bigl(\exp^\perp(y, \xi)\bigr) = |\xi|^2.
\end{equation}
Consequently,  $\Psi(x) = d(x, Y)^2$ is smooth on all of $\CH^2$,  including along
$Y$.

\subsection{The degree-one pinch}

Fix $g>h$.  There is an orientation-preserving degree-one pinch map
\[
 P:S_g\longrightarrow S_h
\]
that collapses the last $g-h$ handles.  On the standard generators,  it induces
the epimorphism
\begin{equation}\label{eq:pinch}
 P_*(a_i) = a_i, \quad P_*(b_i) = b_i\quad(i\le h), 
 \qquad
 P_*(a_i) = P_*(b_i) = 1\quad(i>h).
\end{equation}
Define
\begin{equation}\label{eq:rho-pullback}
 \rho = \rho_0\circ P_*:\pi_1(S_g)\longrightarrow\PU(2, 1).
\end{equation}
This is the representation that is going to be the required one as in Theorem \ref{thm:main}.

\begin{proposition}\label{prop:rho-properties}
The representation $\rho$ in \eqref{eq:rho-pullback} is reductive,  irreducible,  and nonfaithful. Further it satisfies
\[
 \tau(\rho) = 2h-2-\frac{2d}{3}.
\]
When $3\nmid d$,  its Toledo invariant is nonintegral.
\end{proposition}

\begin{proof}
Since $P_*$ is onto, 
\[
 \im\rho = \im\rho_0.
\]
The representation $\rho_0$ is convex-cocompact,  hence discrete,  faithful, 
nonelementary,  and reductive.  We prove its irreducibility directly.  If its
image fixed a point of $\CH^2$,  it would be a discrete subgroup of a compact
stabilizer and hence finite,  contradicting faithfulness.  A fixed point of
$\partial\CH^2$ is likewise incompatible with nonelementary convex
cocompactness.

It remains to rule out preservation of a complex geodesic $C$.  The kernel of
the restriction homomorphism
\[
 \operatorname{Stab}_{\PU(2, 1)}(C)
 \longrightarrow\operatorname{Isom}^+(C)
\]
is compact.  Its intersection with the discrete group $\im\rho_0$ is finite
and is trivial because $\im\rho_0\cong\pi_1(S_h)$ is torsion-free.  The restricted homomorphism is proper because its kernel is compact; therefore the
induced action on $C$ is discrete and faithful.  Further,  an orbit in the totally geodesic $C$ has the same ambient
distances as it has in $C$; hence it is also
convex-cocompact.  A convex-cocompact Fuchsian group isomorphic to
a closed surface group is cocompact; otherwise,  its compact convex core has nonempty boundary and a free fundamental group.  Therefore,  its Toledo invariant has absolute value $2h-2$ \cite{Toledo1989}.  This contradicts
\[
 0 < \tau(\rho_0)  =  2h-2-\frac{2d}{3} < 2h-2.
\]
All this proves  $\rho_0$ is irreducible.  Reductivity and irreducibility depend only on
the image and therefore  $\rho$ is also reductive and irreducible.

On the other hand,  $P_*$ kills the fundamental groups carried by the collapsed handles,  so $\rho$ is not faithful.  Lemma~\ref{lem:naturality} and $\deg P = 1$ give the Toledo formula.  Finally, 

$$\tau(\rho) = \frac23(3h-3-d), $$
which is not an integer when $3\nmid d$.
\end{proof}

\begin{remark}
The representation $\rho$ is not convex-cocompact as a representation of
$\pi_1(S_g)$ as its orbit map cannot be a quasi-isometric embedding because its
kernel is nontrivial.  Its image,  however,  is the same
convex-cocompact subgroup as that of $\rho_0$.  
\end{remark}

\section{A quantitative trapping lemma,  including branches}
In  \cite[Proposition~4.5]{BronsteinHadamard} author proves 
the strong squared-distance subharmonicity inequality on an immersed minimal submanifold in its setting.  

Here we give a self-contained pointwise
reformulation which is suited to our present application.

\begin{lemma}\label{lem:tube}
Let $X$ be a Hadamard manifold with $K_X \le -1$,  and let
$Y^2 \subset  X$ be a complete embedded minimal surface such that
\[
 \sup_{y \in Y}\|\II_Y\|_{\mathrm{op}} \le \eta < \tfrac12, 
 \qquad
 \exp^\perp : NY \longrightarrow X \ \text{is a diffeomorphism}.
\]
Put $\Psi  =  d(\, \cdot\, , Y)^2$.  Then $\Psi \in C^\infty(X)$ and,  for every
$R > 0$,  there is $c_R > 0$ such that,  whenever $d(x,  Y) \le R$ and
$L\subset T_xX$ is a two-plane, 
\begin{equation}\label{eq:tube-estimate}
 \tr_L\Hess\Psi \ge c_R\Psi(x).
\end{equation}
Here $\tr_L$ denotes the trace of the restriction to $L$.
\end{lemma}

\begin{proof}
Under the inverse exponential map,  $\Psi(y,  v) = |v|^2$,  so $\Psi$ is
smooth.  Fix $x \notin Y$,  and write $r  =  d(x,  Y)$ and $n  =  \dim X$,  and let $\nu$ be
the unit normal at the foot point that points toward $x$.  Let
$\mu_1 \le \cdots \le \mu_{n-1}$ be the principal curvatures,  with normal
$\nabla r$,  of the level hypersurface $N_r Y$.  Since $Y$ is minimal,  the two
eigenvalues of its shape operator $A_\nu$ are $-a,  a$ for some
$0 \le a \le \eta$.

Put $t  =  \tanh r$ and $q_r(s)  = (t+s)/(1+st)$.  \cite[Corollary~3.7]{BronsteinHadamard} gives
\begin{align}
 \mu_1 & \ge q_r(-a) \ge q_r(-\eta), \label{eq:mu-one}\\
 \mu_1 + \mu_2
 & \ge q_r(-a) + q_r(a)
  =  \frac{2t(1-a^2)}{1-a^2t^2}
 \ge 2(1-\eta^2) t.\label{eq:mu-two}
\end{align}
For the points not in $Y$,  we have
\begin{equation}\label{eq:hessian-distance}
 \Hess(r^2)  =  2\, \dd r\otimes\dd r + 2r\, \Hess r, 
\end{equation}
therefore  its eigenvalues are $2,  2r \mu_1,  \ldots,  2r \mu_{n-1}$.  Now we apply the min--max
principle,  to get the trace on any two-plane is at least the sum of the two least
eigenvalues.  It is therefore bounded below by
\begin{equation}\label{eq:mixed-min}
 \min \left\{2r(\mu_1 + \mu_2), \ 2 + 2r\mu_1\right\}.
\end{equation}
We remark here that it also treats mixed radial--tube planes.

By \eqref{eq:mu-two},  the first term in \eqref{eq:mixed-min} is at least
\[
 4(1-\eta^2)r\tanh r
 \ge 4(1-\eta^2)\frac{\tanh R}{R}r^2.
\]
Here we used that $r \mapsto \tanh r/r$ is decreasing on $(0,  \infty)$.
Moreover,  $q_r(-\eta) \ge -\eta$ and therefore it is nonnegative for
$r \ge \operatorname{arctanh} \eta$.  Hence \eqref{eq:mu-one} implies
\[
 2 + 2r \mu_1 \ge 2 \delta_\eta, 
 \qquad
 \delta_\eta  =  1-\eta\operatorname{arctanh} \eta > 0;
\]
in the  range $0 < r \le \operatorname{arctanh} \eta$,  we can use
$\mu_1 \ge -\eta$.  Since $r \le  R$,  we may take
\begin{equation}\label{eq:cR}
 c_R  =  \min\left\{
 4(1-\eta^2)\frac{\tanh R}{R}, 
 \frac{2\delta_\eta}{R^2}
 \right\} > 0.
\end{equation}
At a point $x \in Y$,  $\Hess\Psi$ is twice the orthogonal projection onto
$N_x Y$,  while the right-hand side of \eqref{eq:tube-estimate} vanishes.  This
completes the proof.
\end{proof}

The next statement we write in a form that is needed for the energy problem.  We use the sign convention
$$ \Delta = \tr\Hess;$$
in particular,  $\Delta u\le0$ at a local maximum of a smooth function.

The below lemma,  we name,   as Branched trapping. 
\begin{lemma}\label{lem:trapping}
Let $X$ and $Y$ satisfy the hypotheses of Lemma~\ref{lem:tube} and  $M$
be a closed connected Riemann surface.  Let
$\sigma : \pi_1(M) \to \operatorname{Isom}(X)$ have image,  that preserve $Y$ and 
$$ f : \widetilde M \longrightarrow  X$$ 
be a nonconstant $\sigma$-equivariant harmonic map which is weakly conformal.
Then $f(\widetilde M) \subset  Y$.
\end{lemma}

\begin{proof}
Because $Y$ is $\sigma(\pi_1 M)$-in variant,  the function

$$u = \Psi\circ f = d(f(\, \cdot\,  ),  Y)^2$$
is $\pi_1(M)$-invariant and descends to a smooth nonnegative function on the
compact surface $M$.  Set
$$R_0  =  \max\left\{1,  \max_M \sqrt{u} \right\}, 
 \qquad e(f) = \frac12|\dd f|^2.$$

At a point where $e(f) > 0$,  choose an orthonormal frame $e_1,  e_2$ on $M$.
Weak conformality therefore gives
\[
 \dd f(e_i)  =  \lambda E_i, 
\]
where $E_1,  E_2$ are orthonormal in a two-plane $L \subset TX$ and
$\lambda^2  =  e(f)$.  Since $f$ is harmonic,  the composition formula and
Lemma~\ref{lem:tube} gives
\begin{align}
 \Delta_Mu
 &  =  \sum_{i = 1}^{2}\Hess\Psi(\dd f(e_i),  \dd f(e_i))
   + \dd\Psi(\tau(f))\notag\\
 &  = \lambda^2 \tr_L \Hess \Psi
  \ge c_{R_0} e(f)u.\label{eq:subharmonic}
\end{align}
If $\dd f  =  0$,  the same composition formula gives $\Delta_M u = 0$,  so
inequality~\eqref{eq:subharmonic} holds pointwise there as well.  Thus
$\Delta_M u \ge 0$ on the closed connected surface $M$,  and the maximum principle
makes $u$ constant.  If this constant were positive, 
inequality~\eqref{eq:subharmonic} would force
$e(f) \equiv 0$,  contrary to the assumption that $f$ is nonconstant.  Hence
$u\equiv0$,  which is the assertion.
\end{proof}

\begin{remark}
The use of $r^2$,  rather than $r$,  is convenient at points of $Y$: it is smooth
under the normal-exponential diffeomorphism.
\end{remark}

\section{Proof that the energy has no critical point}
We now prove Theorem~\ref{thm:main}.  

\begin{proof}[Proof of Theorem~\ref{thm:main}]
Suppose,  for contradiction,  that $J \in \Teich(S_g)$ is a critical point of
$E_\rho$.  By the equivalences in ~\eqref{eq:critical-equivalence},  the
Corlette harmonic map
$$f  =  f_{\rho, J} : (\widetilde S_g,  J)\longrightarrow\CH^2$$
is weakly conformal and is therefore a possibly branched minimal map.  

It is non-constant; otherwise $\im\rho$ would fix its value,  contrary to irreducibility.

The group $\im \rho  =  \im \rho_0$ preserves the almost-Fuchsian plane
$Y  =  F(\widetilde S_h)$.   Lemma~\ref{lem:trapping} gives 
\begin{equation}\label{eq:image-in-Y}
 f(\widetilde S_g) \subset  Y.
\end{equation}
Since $F : \widetilde S_h \to  Y$ is an embedding,  define
$$H  =  F^{-1} \circ f : \widetilde S_g \longrightarrow \widetilde S_h.$$
injectivity of $F$ and from the equivariance,  we get 
$$ F(H(\gamma x))  =  f(\gamma x)  =  \rho_0(P_*(\gamma))f(x)
   =  F(P_*(\gamma)H(x)).$$
Therefore
\begin{equation}\label{eq:H-equivariance}
 H(\gamma x)  =  P_*(\gamma)H(x), 
 \qquad \gamma\in\pi_1(S_g).
\end{equation}
This implies that $H$ descends to a smooth map
\begin{equation}\label{eq:Hbar}
 \overline H:(S_g,  J)\longrightarrow(S_h,  J_0)
\end{equation}
which induces $P_*$ on fundamental groups after compatible basepoint choices.  Now we equip $\widetilde S_h$ with the induced metric $g_Y  =  F^*g_{\CH^2}$.  For each
$\gamma \in \pi_1(S_h)$,   since metric is invariant,  we get 
$$\gamma^*g_Y = (F\circ\gamma)^*g_{\CH^2}
  = (\rho_0(\gamma)\circ F)^*g_{\CH^2} = g_Y.$$
Thus $g_Y$ descends to a metric on $S_h$.  Because $F$ is holomorphic,  the
conformal class of $g_Y$ is $J_0$.  Since $F$ is an isometric immersion for
$g_Y$ and $f  =  F\circ H$ is weakly conformal,  $H$ is weakly conformal.

For a local orthonormal frame $e_1, e_2$ on the domain,  the composition
formula for tension fields is
\begin{equation}\label{eq:tension-composition}
 \tau(F\circ H)  =  \dd F(\tau(H)) +
 \sum_{i = 1}^2 \II_F\bigl(\dd H(e_i),  \dd H(e_i)\bigr).
\end{equation}
Where $\dd H \ne 0$,   choose orthonormal frame with
$\dd H(e_i)  =  \lambda E_i$ and $E_1,  E_2$ orthonormal.  Then the last term is
$\lambda^2(\II_F(E_1,  E_1) + \II_F(E_2,  E_2))$ and it vanishes because $F$ is
minimal.  Where $\dd H  =  0$,  it vanishes directly.  Since $F\circ H  =  f$ is
harmonic and $\dd F$ is injective,  equation~\eqref{eq:tension-composition}
implies $$\tau(H) = 0.$$
Thus $\overline H$ is a nonconstant harmonic weakly conformal map between
Riemann surfaces.  Its branch set is discrete,  its complement is connected, 
and the orientation sign is constant there.  Therefore  $\overline H$ is globally holomorphic or globally antiholomorphic.

The closed surfaces $S_g$ and $S_h$ are aspherical.  Because
$\overline H_* = P_*$,  the maps $\overline H$ and $P$ are homotopic and
$$\deg\overline H = \deg P = 1.$$

An antiholomorphic nonconstant map has negative signed degree,  so
$\overline H$ is holomorphic. Thus $\overline H$ is unramified and bijective,  hence biholomorphic.
Consequently $g  =  h$,  contradicting $g > h$.  Therefore $E_\rho$ has no critical
point.
\end{proof}

\section{About Forgetful map}\label{rem:fractional-warning}
Let $\widehat{\mathcal M}^{\mathrm{br,  irr}}_{g,  \tau}$ be the set of triples
$(J,  \sigma,  f)$ with the following properties: 
\begin{itemize}
\item $J$ is an orientation-compatible complex structure on the marked $S_g$;
\item $\sigma :\pi_1(S_g) \to \PU(2,  1)$ is irreducible and reductive with
$\tau(\sigma)  =  \tau$; and
\item $f:(\widetilde S_g, J)\longrightarrow\CH^2$
is a nonconstant,  $\sigma$-equivariant,  harmonic,  weakly conformal map.
\end{itemize}
Here,  for $f$,  the branch points are allowed.  If
$\phi \in \operatorname{Diff}_0^+(S_g)$,  $\widetilde\phi$ is a lift,  and $A \in \PU(2,  1)$,  then
\begin{equation}\label{eq:minimal-moduli-equivalence}
 (J,  \sigma, f) \sim
 \bigl(\phi^*J,  \ A(\sigma \circ \phi_*)A^{-1},  \
  A \circ f \circ\widetilde \phi\bigr).
\end{equation}
Put
$$\mathcal M^{\mathrm{br,  irr}}_{g,  \tau}  =  \widehat{\mathcal M}^{\mathrm{br,  irr}}_{g,  \tau}/\!\sim.$$
Thus $J$ remains marked; reparametrizations isotopic to the identity have been divided out.

Let $\cX_\tau^{\mathrm{irr}}(S_g,  \PU(2,  1))$ be the set of
$\PU(2,  1)$-conjugacy classes $[\sigma]$ of irreducible reductive
representations $\sigma : \pi_1(S_g) \to \PU(2,  1)$ with
$\tau(\sigma)  =  \tau$.   The forgetful map we define here is
\begin{equation}\label{eq:forgetful-map}
 \Phi_{g, \tau} : \mathcal M^{\mathrm{br,  irr}}_{g,  \tau}
 \longrightarrow \cX_\tau^{\mathrm{irr}}(S_g,  \PU(2,  1)), 
 \qquad [(J,  \sigma,  f)] \longmapsto [\sigma].
\end{equation}
It is well defined because $\phi \in \operatorname{Diff}_0^+(S_g)$ induces an
inner automorphism of $\pi_1(S_g)$. See  \cite[Section~6.3]{LoftinMcIntosh}.

\begin{corollary}\label{cor:nonsurjective}
For the triples $(g,  h,  d)$ in Theorem~\ref{thm:main},  the map
$\Phi_{g,  2h-2-2d/3}$ is not surjective.  More precisely, 
$$[\rho_{g,  h,  d}] \notin \operatorname{im}\Phi_{g,  2h-2-2d/3}.$$
\end{corollary}

\begin{proof}
By uniqueness of the equivariant harmonic map and the equivalences
in~\eqref{eq:critical-equivalence},  for a triple above,  a character $[\sigma]$
exists precisely when $E_\sigma$ has a critical point. Now  applying 
Theorem~\ref{thm:main},  we get a contradiction. 
\end{proof}

\end{document}